\documentclass{amsart}
\usepackage{bbm}
\usepackage{verbatim}
\usepackage[latin1]{inputenc}

\def\N{{\mathbb N}}
\def\R{{\mathbb R}}

\def\P{{\mathbb P}}
\def\E{{\mathbb E}}

\def\cal{\mathcal}
\def\eps{\varepsilon}
\def\etal{{\em et al.}}

\newcommand\ind[1]{\mathbbm{1}_{\{#1\}}}

\newtheorem{thm}{Theorem}
\newtheorem{prop}{Proposition}
\newtheorem*{lemma}{Lemma}
\newtheorem{corollary}{Corollary}

\theoremstyle{definition}
\newtheorem{definition}{Definition}

\title[Occupancy Schemes Associated to Yule Processes]{Occupancy Schemes Associated to Yule Processes}

\author[Ph. Robert]{Philippe Robert}
\author[F. Simatos]{Florian Simatos}
\address[Ph. Robert, F. Simatos]{INRIA Paris --- Rocquencourt, Domaine de Voluceau, BP 105,  
78153   Le Chesnay, France. } 
\email{Philippe.Robert@inria.fr}
\email{Florian.Simatos@inria.fr}
\urladdr{http://www-rocq.inria.fr/\string~robert}

\date{\today}

\begin{document}

\bibliographystyle{plain}

\begin{abstract}
An occupancy problem with  an infinite number of bins and a  random probability vector for
the locations of the balls is considered.  The respective sizes of bins are related to the
split times of  a Yule process.  The  asymptotic behavior of the landscape  of first empty
bins, i.e.,  the  set of corresponding indices represented by  point processes, is analyzed
and    convergences    in    distribution    to   mixed   Poisson    processes    are
established. Additionally, the influence of the random environment, the random probability
vector, is analyzed. It  is represented by two main components: an  i.i.d.\ sequence and a
fixed random variable.  Each of these  components has a specific impact on the qualitative
behavior of the  stochastic model. It is shown  in particular that for some  values of the
parameters, some  rare events,  which are  identified, play an  important role  on average
values of the number of empty bins in some regions.
\end{abstract}

\maketitle

\bigskip

\hrule

\vspace{-1cm}

\tableofcontents

\vspace{-1cm}

\hrule

\bigskip

\section{Introduction}
Occupancy schemes in terms of bins and  balls offer a very  flexible and elegant  way to
formulate  various problems  in  computer  science, biology  and  applied mathematics  for
example.  One of  the earliest models investigated in the  literature consists in throwing
$m$ balls  at random into $n$  identical bins. Asymptotic behavior  of occupancy variables
have been  analyzed when $n$  grows to  infinity, with different  scalings in $n$  for the
variable   $m$.   The   books   by  Johnson   and   Kotz~\cite{Johnson77:0}  and   Kolchin
\etal~\cite{Kolchin:01} are  classical references  on this topic.   See also  Chapter~6 of
Barbour \etal~\cite{Barbour:01} for a recent presentation of these problems.

An extension of these models is when there is an infinite number of bins and a probability
vector $(p_n)$  on $\N$ describing  the way balls  are sent: for  $n\geq 0$, $p_n$  is the
probability that a ball  is sent into the $n$th bin.  In one of  the first studies in this
setting,  Karlin~\cite{Karlin67:0}  analyzed the  asymptotic  behavior  of  the number  of
occupied bins.  More recently Hwang  and Janson~\cite{Hwang07:0} proves in a quite general
framework central limit results for these  quantities.  In this setting,
some  additional variables  are also  of interest  like the  sets of  indices of
occupied or  empty bins,  adding a  geometric component to  these problems.   For specific
probability  vectors  $(p_n)$  Cs\'aki  and  F\"oldes~\cite{Csaki76:0}  and
Flajolet and Martin~\cite{Flajolet85:0} investigated the  index of the first empty bin.
See the recent survey Gnedin  \etal~\cite{Gnedin07:0} for more references on the occupancy
problem with infinitely many bins.

A  further extension  of  these stochastic  models  consists in  considering {\em  random}
probability vectors.  Gnedin~\cite{Gnedin04:0} (and  subsequent papers) analyzed  the case
where  $(p_n)$ decays geometrically  fast according  to some  random variables,  i.e.,  for
$n\geq1$, $p_n=\prod_{i=1}^{n-1} Y_i(1-Y_n)$ where  $(Y_i)$ are i.i.d.\ random variables on
$(0,1)$.  Various  asymptotic results  on the number  of occupied  bins in this  case have
been obtained. The random vector can be seen as a ``random environment'' for the bins and
balls problem, it complicates significantly the asymptotic results in some cases. In
particular, the  indices of the urns in which the balls fall are no longer independent random
variables as in the deterministic case. 

The general goal of  this paper is to investigate in detail  the impact of this randomness
for  a   bins  and  balls  problem  associated   to  a  Yule  process,   see  Athreya  and
Ney~\cite{Athreya:04} for the definition of a Yule process. This (quite natural) stochastic model has its origin in
network modeling,  see Simatos \etal~\cite{Simatos08:0}  for a detailed  presentation.  It
can be  described as follows:  the non-decreasing sequence  $(t_n)$ of split times  of the
Yule  process  defines  the  bins,  the   $n$th  bin,  $n\geq  1$,  being  the  interval
$(t_{n-1},t_n]$. The  locations of  balls are represented  by independent exponential  random variables
with parameter $\rho$. The main problem investigated here concerns the asymptotic
description of the set of  indices of   first empty bins when the number of
balls goes to infinity. Mathematically, it is  formulated as a convergence in distribution
of  rescaled point processes having Dirac masses at the indices of empty bins.

For $n\geq  1$, if $P_n$ is the  probability that a ball  falls into the $n$th  bin, it is
easily seen  that, for a  large $n$,  $P_n$ has a power law decay, it can be  expressed as
$V  E_n/n^{\rho+1}$ where 
$(E_n)$  are  i.i.d.\  exponential  random  variables  with parameter  $1$  and  $V$  some
independent random  variable related to the limit  of a martingale. The  randomness of the
probability vector $(P_n)$ has two components: one which is a part of an i.i.d.\ sequence,
changing  from one bin  to another,  and the  other being  ``fixed once  for all''  inducing a
dependency  structure.   As  it  will  be  seen, the  two  components  have  separately  a
significant impact on the qualitative behavior of this model.

\subsection*{Convergence in Distribution and Rare Events}
Because the  variables $(E_n)$ can be  arbitrarily small with  positive probability, empty
bins  are  likely  to  be  created  earlier  (i.e.,  with  smaller  indices)  than  for  a
deterministic probability vector with the same power  law decay. It is shown in fact that,
for the  convergence in  distribution, the first  empty bins  occur around indices  of the
order of $n^{1/(\rho+2)}$ instead of $(n/\log n)^{1/(1+\rho)}$ in the deterministic case.

The variable  $V$ has a more  subtle impact, when $\rho>1$  it is shown that,  due to some
heavy  tail property  of $V^{-1}$,  rare  events affect  the asymptotic  behavior of  {\em
  averages} of  some of the  characteristics.  For $\alpha  \in [1/(2\rho+1),1/(\rho+2))$,
  despite that the number of empty bins with indices of order $n^\alpha$ converges {\em in
    distribution}  to  $0$,  the  corresponding  average  converges  to  $+\infty$.   When
  $\rho<1$, the average  is converging to $0$.  A phase  transition phenomenon at $\rho=1$
  has been identified through simulations in a related context, communication networks, in
  Saddi  and Guillemin~\cite{Saddi:01}.  It  is not  apparent  as long  as convergence  in
  distribution is concerned but it shows  up when average quantities are considered.  This
  phenomenon is due to rare events related  to the total size of the $\lfloor \rho\rfloor$
  first  bins: On these  events, the  indices of  the first  empty bins  are of  the order
  $n^{1/(2\rho+1)}\ll n^{1/(\rho+2)}$  and a  lot them are  created at this  occasion. See
  Proposition~\ref{kr}  and  Corollary~\ref{corolkr}  for  a  precise  statement  of  this
  result. Concerning the  generality of the results obtained, it is  believed that some of
  them hold in  a more general setting, for the underlying  branching process for example,
  see Section~\ref{secgen}.

\subsection*{Point Processes}
Technically, one mainly uses point processes on $\R_+$ to describe the asymptotic behavior
of the indices of the first empty bins and not only  the index of the first
one (or  the subsequent ones) as  it is usually the  case in the literature.  It turns out
that it is quite appropriate in our setting  to get a full description of the set of
the first empty bins  and, moreover, it reduces the technicalities of  some of the proofs.
One of the  arguments for the proofs of the convergence  results is a simple
convergence result of two-dimensional point processes to Poisson point process with some
intensity measure. 
A one-dimensional equivalent  of this point of view  is implicit in most of  the papers of
the  literature, in  Hwang  and  Janson~\cite{Hwang07:0} in  particular.   See Robert  and
Simatos~\cite{Robert:IP} for  a presentation of  an extension of  this approach in  a more
general framework.

The paper  is organized as  follows.  Section~\ref{UBsec} introduces the  stochastic model
investigated.   The main  results concerning  convergence  of related  point processes  in
$\R_+^2$ are  presented in Section~\ref{secpp}.   Convergence results for the  indices of
empty   bins   are  proved   in   Section~\ref{seclimit}.  Section~\ref{secrare}
investigates in detail the case $\rho\geq 1$. Section~\ref{secgen} presents some possible
extensions.  

\section{A Bins and Balls Problem in Random Environment}\label{UBsec}
\noindent
The stochastic model is described in detail and some notations are introduced.
\subsection*{The Bins}
Let $(E_i)$ be a sequence of i.i.d.\ exponential random variables with parameter $1$. Define the
non-decreasing sequence  $(t_n)$ by, for $n\geq 1$,
\[
t_n=\sum_{i=1}^n \frac{1}{i}E_i.
\]
It is easy to check that for $x\geq 0$,
\begin{equation}\label{eqtn}
\P( t_n\leq x)=\P(\max(E_1, E_2, \ldots, E_n)\leq x)= (1-e^{-x})^n.
\end{equation}
The $n$th bin will be identified by the interval $(t_{n-1},t_n]$.

If $H_n=1+1/2+\cdots+1/n$ is the $n$th harmonic number, since $(t_n-H_n)$ is a
square integrable martingale whose increasing process is given by 
\[
\E\left((t_n-H_n)^2\right)=\sum_{i=1}^n \frac{1}{i^2},
\]
then $(M_n)\stackrel{\text{def.}}{=}(t_n-\log n)$ is almost surely converging to some
finite random variable $M_\infty$. See Neveu~\cite{Neveu:18} or
Williams~\cite{Williams91:0}. By using Equation~\eqref{eqtn}, it is not difficult to get
that the distribution of $M_\infty$ is given by  
\begin{equation}\label{Minfty}
\P(M_\infty\leq x)=\exp\left(-e^{-x}\right), \quad x\in\R.
\end{equation}
An alternative description of the sequence $(t_n)$ is provided by the split times of a
Yule (branching) process starting with one individual. See Athreya and Ney~\cite{Athreya:04}.

\subsection*{The Balls}
The locations of the balls are given by an independent sequence $(B_j)$ of
i.i.d.\ exponential random variables with parameter $\rho$ for some $\rho>0$.

Conditionally on the point process $(t_n)$ associated with the location of bins, the
probability that a given ball falls into the $n$th bin $(t_{n-1},t_n]$ is given by
\[
P_n=\P\left[B_1\in(t_{n-1},t_n]\left|\rule{0mm}{4mm} (t_n)\right]\right.
=e^{-\rho t_{n-1}}-e^{-\rho t_{n}}=e^{-\rho t_{n-1}}\left(1-e^{-\rho E_{n}/n}\right).
\]
This quantity can be rewritten as 
\begin{equation}\label{Pn}
P_n= \frac{1}{n^{\rho+1}} W_n^\rho Z_n,\text{ with }
 Z_n=n\left(1-e^{-\rho E_n/n}\right)
\text{ and } W_n=e^{-M_{n-1}}.
\end{equation}
The variables $W_n$ and $Z_n$ are independent random variables with different behavior. 
\begin{enumerate}
\item The variables $(Z_n)$ are independent and
converge in distribution to an exponentially distributed random variable with parameter
$1/\rho$. 
\item The random variables $(W_n)$ converge almost surely to the finite random variable
$W_\infty=\exp(-M_\infty)$ which is exponentially distributed with parameter $1$.
\end{enumerate}
This suggests an asymptotic representation of the sequence  $(P_n)$ as
\begin{equation}\label{asympt}
P_n\sim\frac{1}{n^{\rho+1}} W_\infty^\rho \widetilde{E}_n,
\end{equation}
where $(\widetilde{E}_n)$ is an i.i.d.\ sequence of exponential random variables with mean
$\rho$  independent of $W_\infty$.   The sequence  $(P_n)$ has  a power  law decay  with a
random  coefficient consisting  of  the product  of  two terms:  a  fixed random  variable
$W_\infty^\rho$ and  the other being  an element  of an i.i.d.\  sequence.  As it  will be
seen, these two terms have a significant  impact on the bins and balls problem studied in
this paper.

\section{Convergence of Point Processes}\label{secpp}
One of the main result,  Theorem~\ref{Res} in the next section, which  establishes convergence
results for  the indices of  the first  empty bins is  closely related to  the asymptotic
behavior of the point process $\{(i/n^{1/(2+\rho)},nP_i), i\geq 1\}$ on $\R_+^2$. For this
reason, some results  on convergence of point processes in $\R_+^2$  are first proved. The
point process  associated to the $(nP_i)$  appears quite naturally, especially  in view of
the  Poisson transform used  in the  proof of  Theorem~\ref{Res}. This  is also  a central
variable in Hwang and Janson~\cite{Hwang07:0} in some cases. 

An important tool to  study point processes in $\R_+^d$ for some  $d\geq 1$ is the Laplace
transform:  If ${\cal  N}{=}\{t_n, n\geq  1)$ is  a point  process and  $f$ a  function in
$C_c^+(\R_+^d)$, the set  of non-negative continuous functions with  a compact support, it
is defined as $\E(\exp(-{\cal N}(f)))$, where
\[
{\cal N}(f)\stackrel{\text{def.}}{=}-\sum_{n\geq 1} f(t_n).
\]
This functional uniquely determines the distribution of ${\cal N}$ and it is a key tool to
establish convergence results.   See Neveu~\cite{Neveu:12} and Dawson~\cite{Dawson:16} for
a comprehensive  presentation of these questions.  In the following, the quantity  ${\cal
 N}(A)$  denotes the number of $t_n$'s in the subset $A$ of $\R_+^d$.

The main results of this section establish convergence in distribution to mixed Poisson
point processes, i.e., distributed as a Poisson point process with a parameter which is a
random variable.   A natural tool in  this domain is the  Chen-Stein approach which
gives the convergence in distribution and, generally, quite good bounds on the convergence
rate. See Chapter~10 of Barbour \etal~\cite{Barbour:01} for example. This has been used in
Simatos \etal~\cite{Simatos08:0}, when the  probability vector is deterministic.  For some
of the results  of this section, this approach can probably  also be used.  Unfortunately,
due to the  almost surely converging sequence $(W_n)$ creating  a dependency structure, it
does not  seem that the main convergence  result, Theorem~\ref{theoPn}, can be  proved in a
simple way  by using Chen-Stein's  method. The main  problem being of conditioning  on the
variable  $W_\infty$ and keeping  at the  same time  upper bounds  on the  total variation
distance converging to $0$.

\subsection*{Condition~C}
A sequence of independent random variables $(X_i)$ satisfies Condition~C if there exist
some $\alpha>0$ and $\eta>0$ such that, for all $i\geq 1$, 
\begin{equation}\label{cond}
|\P(X_i\leq x)-\alpha x|\leq C x^2, \text{ when } 0\leq x\leq \eta.
\end{equation}
The following proposition is a preliminary result that will be used to prove the
main convergence results for the indices of the first empty bins. 
\begin{prop}[Convergence to a Poisson process]\label{cvPoispp}
For $\delta>0$ and $n\geq 1$, let ${\cal P}_n$ be the point process on $\R_+^2$ defined by 
\[
{\cal P}_n\stackrel{\text{def.}}{=}\left\{\left(\frac{i}{n^{1/(\delta+1)}},\frac{n}{i^{\delta}}X_i\right), i\geq 1\right\}, 
\]
where $(X_i)$ a  sequence of non-negative independent random  variables satisfying
Condition~C. Then  the sequence of 
point processes $({\cal P}_n)$ converges in distribution to a Poisson point process ${\cal
  P}$ in $\R_+^2$ with intensity  measure $x^{\delta}dx\,dy$ on $\R_+^2$. In particular,
its Laplace transform is given by
\begin{equation}\label{Laplace}
\E(\exp[-{\cal   P}(f)])=\exp\left(-\alpha\int_{\R_+^2}\left(1-e^{-f(x,y)}\right)x^{\delta}\,dx\,dy\right),
\quad f\in C_c^+(\R_+^2).
\end{equation}
\end{prop}
\noindent
See Robert~\cite{Robert:08} for the definition and the main properties of Poisson
processes in general state spaces. 
\noindent
\begin{proof}
There exists some $\eta_0>0$ such that $\P(X_i\leq x)\leq 2\alpha x$ for $0\leq x\leq
\eta_0$ and all $i\geq 1$.
Let $f\in C_c^+(\R_+^2)$ be such that $f$ is differentiable with respect to the second
variable. There is some $K>0$ so that the support of $f$ is included in
$[0,K]\times[0,K]$,  define $g(x,y)=1-\exp(-f(x,y))$, then by independence of the
variables $X_i$, $i\geq 1$, 
\[
\log \E\left(e^{-{\cal   P}_n(f)}\right)=\sum_{i=1}^{+\infty}
\log\left(1-\E\left[g\left(\frac{i}{n^{1/(\delta+1)}},\frac{n}{i^{\delta}}X_i\right)\right]\right).
\]
Since
\[
\E\left[g\left(\frac{i}{n^{1/(\delta+1)}},\frac{n}{i^{\delta}}X_i\right)\right]\leq 
\P\left(X_i\leq K\frac{i^{\delta}}{n}\right)\ind{i\leq K n^{1/(\delta+1)}},
\]
the elementary inequality $|\log(1{-}y)+y|\leq 
3y^2/2$ valid for $0\leq y\leq 1/2$ shows that there exists some $n_0\geq 1$ such that
\begin{multline*}
\left|\log \E\left(e^{-{\cal P}_n(f)}\right)
+\sum_{i=1}^{+\infty}\E\left[g\left(\frac{i}{n^{1/(\delta+1)}},\frac{n}{i^{\delta}}X_i\right)\right]\right|\\
\leq \frac{6(\alpha K)^2}{n^2}\sum_{i=1}^{\lfloor Kn^{1/(\delta+1)}  \rfloor}i^{2\delta}
\leq {6}\alpha^2 K^{2\delta+3}\frac{1}{n^{1/(\delta+1)}}
\end{multline*}
holds for $n\geq n_0$.
It is therefore enough to study the asymptotics of the series of the left hand side of the
above inequality. For $x\geq 0$, by using Fubini's Theorem, one gets
\[
\E\left(g\left(x,\frac{n}{i^{\delta}}X_i\right)\right)=
-\int_0^{+\infty}\frac{\partial g}{\partial y} (x,y)\P\left(X_i\leq y i^{\delta}/n\right)\,dy.
\]
By using again Condition~C, one obtains that the log of
the Laplace transform of ${\cal P}_n$ has the same asymptotic behavior as
\[
-\alpha \frac{1}{n^{1/(\delta+1)}} \sum_{i=1}^{+\infty}\int_0^{+\infty}\frac{\partial
  g}{\partial y} \left(\frac{i}{n^{1/(\delta+1)}},y\right) y
\left(\frac{i}{n^{1/(\delta+1)}}\right)^{\delta}\,dy 
\]
which is a Riemann sum converging to
\[
-\alpha \int_{\R_+^2}\frac{\partial g}{\partial y} (x,y)yx^{\delta}\,dxdy=\alpha \int_{\R_+^2}\left(1-e^{-f(x,y)}\right)x^{\delta}\,dxdy.
\]
This shows in particular that for any compact set $H$ of $\R_+^2$, then
\[\sup_{n\geq 1}\E({\cal P}_n(H))<+\infty,\] the sequence $({\cal P}_n)$ is therefore
tight for the weak topology, see Dawson~\cite{Dawson:16}. 

By the convergence result, if ${\cal P}$ is any limiting point of
the sequence $({\cal P}_n)$, for any function $f\in C_c^+(\R_+^2)$ such that $y\to f(x,y)$
is differentiable, then the Laplace 
transform of ${\cal P}$ at $f$ is given by the right hand side of
Equation~\eqref{Laplace}. By density of these functions $f$ in $C_c^+(\R_+^2)$ for the
uniform topology, this implies that ${\cal P}$ is indeed a Poisson point process with
intensity measure $x^{\delta}\,dxdy$ on $\R_+^2$. The proposition is proved. 
\end{proof}
The above  result can be (roughly)  restated as follows: for  the indices of  the order of
$n^{1/(\delta+1)}$,  the  points $nX_i/i^\delta$,  lying  in  some  finite fixed  interval
converge  to an  homogeneous  Poisson point  process.  The following  proposition gives  an
asymptotic  description of  the indices of the points  $nX_i/i^\delta$  but for  indices of  the order  of
$n^{\kappa}$ with $\kappa>1/(\delta+1)$. It shows  that, on finite intervals, these points
accumulate at  rate $n^{(1+\delta)\kappa-1}$ according  to the Lebesgue measure  with some
density.
\begin{prop}[Law of Large Numbers]
If, for $\kappa>1/(1+\delta)$ and for $n\in\N$,  ${\cal P}_n^\kappa$ is the point process
on $\R_+$ defined by
\[
{\cal P}_n^\kappa(f)=\frac{1}{n^{(1+\delta)\kappa-1}}
\sum_{i=1}^{+\infty}f\left(\frac{i}{n^\kappa},\frac{n}{i^\delta}X_i\right), 
\quad f\in C_c^+(\R_+^2),
\]
where $(X_i)$  is a sequence of non-negative independent random  variables satisfying
Condition~C, then the sequence 
$({\cal P}_n^\kappa)$ converges in distribution to the deterministic measure ${\cal
  P}_\infty^\kappa$ defined by  
\[
{\cal P}_\infty^\kappa(f)=\alpha\int_{\R_+^2}f(x,y)x^\delta\,dxdy,\quad f\in C_c^+(\R_+^2). 
\]
\end{prop}
\begin{proof}
Let $f\in  C_c^+(\R_+^2)$ be such  that $f$ is  differentiable with respect to  the second
variable.  As before,  the  convergence result  is proved  for  such a  function $f$,  the
generalization to an arbitrary function $f\in C_c^+(\R_+^2)$ follows the same lines as the
previous proof (relative compactness argument and identification of the limit).   Let  $K>0$   such
that   the  support   of   $f$   is  included   in $[0,K]\times[0,K]$. One has 
\[
\E\left({\cal P}_n^\kappa(f)\right)= -\frac{1}{n^{(1+\delta)\kappa-1}} \sum_{i=1}^{+\infty} \int_0^{+\infty}
\frac{\partial f}{\partial y}\left(\frac{i}{n^\kappa},y\right)\P(X_i\leq yi^\delta/n)\,dy,
\]
as in the previous proof, by using Condition~\eqref{cond}, one gets that
\[
\E\left({\cal P}_n^\kappa(f)\right)\sim -\alpha \frac{1}{n^\kappa} \sum_{i=1}^{+\infty} \int_0^{+\infty}
\frac{\partial f}{\partial y}\left(\frac{i}{n^\kappa},y\right)y\left(\frac{i}{n^\kappa}\right)^\delta\,dy,
\]
therefore,
\[
\lim_{n\to+\infty} \E\left({\cal P}_n^\kappa(f)\right)=
-\alpha\int_0^{+\infty}\int_0^{+\infty} \frac{\partial f}{\partial y}\left(x,y\right)y
x^\delta\,dxdy= \alpha\int_{\R_+^2}f(x,y)x^\delta\,dxdy.
\]
By independence of the $X_i$'s the second moment of the difference
\begin{multline*}
{\cal P}_n^\kappa(f)-\E\left({\cal P}_n^\kappa(f)\right)
\\= -\frac{1}{n^{(1+\delta)\kappa-1}} \sum_{i=1}^{+\infty} \int_0^{+\infty}
\frac{\partial f}{\partial y}\left(\frac{i}{n^\kappa},y\right)\left[\ind{X_i\leq yi^\delta/n}-\P(X_i\leq yi^\delta/n)\right]\,dy,
\end{multline*}
can be expressed as
\begin{align*}
&n^{2((1+\delta)\kappa-1)}\times\E\left(\left[{\cal P}_n^\kappa(f)-\E\left({\cal P}_n^\kappa(f)\right)\right]^2\right)\\
&= \sum_{i=1}^{+\infty} \E\left(\left[\int_0^{+\infty}
\frac{\partial f}{\partial
  y}\left(\frac{i}{n^\kappa},y\right)\left[\ind{X_i\leq yi^\delta/n}-\P(X_i\leq yi^\delta/n)\right]\,dy\right]^2\right)\\
&\leq K\sum_{i=1}^{+\infty} \int_0^{+\infty}
\left[\frac{\partial f}{\partial  y}\left(\frac{i}{n^\kappa},y\right)\right]^2\E\left(\left[\ind{X_i\leq yi^\delta/n}-\P(X_i\leq yi^\delta/n)\right]^2\right)\,dy\\
&\leq K\sum_{i=1}^{+\infty} \int_0^{+\infty}
\left[\frac{\partial f}{\partial  y}\left(\frac{i}{n^\kappa},y\right)\right]^2\P(X_i\leq yi^\delta/n)\,dy,
\end{align*}
by Cauchy-Shwartz's Inequality. The last term is, with the same arguments as for the
asymptotics of $\E\left({\cal P}_n^\kappa(f)\right)$, equivalent to 
\[
K n^{{(1+\delta)\kappa-1}}\times \int_{\R_+^2}
\left[\frac{\partial f}{\partial  y}\left(x,y\right)\right]^2 y x^\delta\,dxdy.
\]
In particular, the sequence $({\cal P}_n^\kappa(f))$ converges in $L_2$ (and therefore in
distribution) to ${\cal  P}_\infty^\kappa(f)$. The proposition is proved. 

The main convergence result can now be established. 
\end{proof}
\begin{thm} \label{theoPn}
If, for $n\geq 1$, ${\cal P}_{n} $ is the point process on $\R_+^2$ defined by 
\[
{\cal P}_{n} =\left\{\left(\frac{i}{n^{1/(\rho+2)}},nP_i\right), i\geq 1\right\},
\]
 then the sequence $({\cal P}_n)$ converges in distribution and the relation
\begin{equation}\label{LapTrans}
\lim_{n \to +\infty} \E \left( e^{-{\cal P}_{n}(f)} \right) = 
\E\left[\exp\left(-\frac{W_\infty^{-\rho}}{\rho}
\int_{\R_+^2}\left(1-e^{-f(x,y)}\right)x^{\rho+1}\,dx\,dy\right)\right]
\end{equation}
holds for any $f\in C_c^+(\R_+^2)$. 
\end{thm}
In other words  the point process ${\cal P}_{n}$ converges in  distribution to a 
mixed Poisson point process:  conditionally  on   $W_\infty$,  it  is  a  Poisson
process  with  intensity measure $W_\infty^{-\rho} x^{\rho+1}\,dx\,dy/\rho$.
\begin{proof}
The proof proceeds in several steps. The main objective of these steps is to decouple the
sequences $(W_i)$ and $(Z_i)$ defining the $(P_i)$ and then to apply Proposition~\ref{cvPoispp}.
\medskip
\paragraph{Step 1}
One defines the sequences
\[
P_i^1=\frac{1}{i^{\rho+1}}\widetilde{W}_{\infty}^\rho \widetilde{Z}_i, \quad i\geq 1, \quad
P_i^2=\frac{1}{i^{\rho+1}}\widetilde{W}_{\beta_n}^\rho \widetilde{Z}_i, \quad i\geq 1,
\]
where $(\beta_n)$ is some sequence of integers converging to $+\infty$.  The  sequences of
random variables $(\widetilde{W}_i,1\leq i\leq +\infty)$ and $(\widetilde{Z}_i)$ are
assumed to be independent and to 
have, respectively, the  same distribution as  $(W_i, 1\leq i\leq +\infty)$ and  $(Z_i)$ defined by
Equation~\eqref{Pn}. Recall that the sequence $(\widetilde{W}_i)$ converges almost  surely
to $\widetilde{W}_\infty$.  These sequences define point processes in the following way, for $j=1$ and $2$,  
\[
{\cal  P}_n^j=\left\{\left(\frac{i}{n^{1/(\rho+2)}},nP_i^j\right), i\geq 1\right\}.
\]
If $f$  is a non-negative continuous  function with compact support  on $\R_+^2$, because,
conditionally  on  $\widetilde{W}_\infty$,   the  variables  $(\widetilde{W}_\infty  Z_i)$
satisfy Condition~C  with $\alpha=\widetilde{W}_\infty^{-\rho}/\rho$,
Proposition~\ref{cvPoispp}, with $\delta=\rho+1$, shows
that
\[{}
\lim_{n\to+\infty} \E\left.\left(e^{-{\cal P}_n^1(f)}\right|\widetilde{W}_\infty\right)=
\exp\left(-\frac{\widetilde{W}_\infty^{-\rho}}{\rho}\int_{\R_+^2}\left(1-e^{-f(x,y)}\right)x^{\rho+1}\,dx\,dy\right).
\]
Because of the  boundedness of these quantities, by Lebesgue's Theorem, the same result
holds for the expected values. Therefore, the sequence $({\cal P}_n^1)$ converges in
distribution to the point process ${\cal   P}$ on $\R_+^2$ whose Laplace transform is
given by Equation~\eqref{LapTrans}.  

Let $K\geq 2$ be such that the support of $f$ is a subset of $[0,K]^2$
and $\eps>0$  .  Since  the limiting  point process ${\cal  P}$ is  almost surely  a Radon
measure, there  exists some $m\in\N$  such that $\P({\cal P}_n^1([0,2K]^2)\geq  m)\leq \eps$
for  all   $n\geq  1$.   By  uniform  continuity,   there  exists  $0<\eta<1/2$   such
that $|f(u)-f(v)|\leq \eps/m$ for $u$, $v\in\R_+^2$ such that $\|u-v\|\leq \eta$. For
$n\geq 1$, if 
\[
{\cal A}\stackrel{\text{def}}{=}\{|{\widetilde{W}_{\beta_n}^\rho}/{\widetilde{W}_{\infty}^\rho}-1|
\geq \eta/2K\}\cup \{{\cal P}_n^1([0,2K]^2)\geq  m\}
\]
then
\begin{multline*}
\left|\E\left(\exp\left[-{\cal P}_n^2(f)\right]\right)-\E\left(\exp\left[-{\cal
      P}_n^1(f)\right]\right)\right|\leq \P({\cal A})\\
+ \E\left(\left(\exp\left[\sum_{i\geq 1}  \left|f\left(\frac{i}{n^{1/(\rho+2)}},\frac{\widetilde{W}_{\beta_n}^\rho}{\widetilde{W}_{\infty}^\rho}nP_i^1\right)
- f\left(\frac{i}{n^{1/(\rho+2)}},nP_i^1\right)\right|\right]-1\right){\mathbbm 1}_{{\cal
    A}^c}\right)\\
\leq \P(|W_{\beta_n}^\rho/W_{\infty}^\rho-1|
\geq \eta/2K)+2\eps,
\end{multline*}
hence, by the almost sure convergence of $(W_n)$ to $W_\infty$, the right hand side of the last relation can be arbitrarily small as $n$ goes to
infinity. One concludes that the sequence $({\cal P}_n^2)$ also converges in distribution to
the  point process ${\cal P}$.

\medskip

\paragraph{Step 2}
For $n\geq 1$, define
\[
 \beta_n = \left \lfloor {n^{1/(\rho+2)}}/{\log n} \right \rfloor,
\]
then it will be shown that the  point processes 
\[
{\cal Q}_n=\left\{\left(\frac{i}{n^{1/(\rho+2)}},nP_i\right), 1\leq i\leq \beta_n\right\} 
\]
converge to the measure identically null. It is sufficient to prove that for any $f\in
C_c^+(\R_+)$, the sequence $({\cal Q}_{n}(f))$ converges in distribution to $0$. 
For a fixed $i$, the sequence $(nP_i)$ converges in distribution to infinity, since $f$ is
continuous with compact support and therefore bounded, one obtains that, in the definition
of ${\cal Q}_n$, it can be assumed that the indices $i$ are restricted to the set
$\{\lceil \rho\rceil, \ldots,\beta_n\}$. 

Let $K$ be such that the support of $f$ is included in $[0,K]^2$, if $u_n=\log\log n$,
for $i\geq \lceil \rho\rceil$, 
\[
\E\left(f(i/n^{1/(\rho+2)},nP_i)\ind{t_{\lfloor \rho\rfloor}\leq u_n}\right)\leq 
\|f\|_{\infty} \P\left(t_{\lfloor \rho\rfloor}\leq u_n, nP_i\leq K\right),
\]
since $P_i=e^{-\rho t_{\lfloor \rho\rfloor}}e^{-\rho(t_{i-1}-t_{\lfloor  \rho\rfloor})}\left(1-e^{-\rho E_i/i}\right)$, 
\begin{multline*}
\E\left(f(i/n^{1/(\rho+2)},nP_i)\ind{t_{\lfloor \rho\rfloor}\leq u_n}\right)\\\leq 
\|f\|_{\infty}\P\left[\left(\frac{1-e^{-\rho E_i/i}}{\rho/i}\right)\leq \frac{i}{\rho} Ke^{\rho
    u_n}e^{\rho(t_{i-1}-t_{\lfloor  \rho\rfloor})}/n\right].
\end{multline*}
By using the elementary inequality, if $E_1$ is exponentially distributed with mean $1$,
\begin{equation}\label{ineq1}
\P\left(\frac{1}{y}\left(1-e^{-y E_1}\right)\leq x\right)\leq e\left(1-e^{-x}\right),\quad y\leq 1, x\geq0,
\end{equation}
one gets that, for $i>\rho$,
\begin{align*}
\E\left(f(i/n^{1/(\rho+2)}, nP_i)\ind{t_{\lfloor \rho\rfloor}\leq u_n}\right)&\\
\leq e&\|f\|_{\infty} \,\E\left(1-\exp\left[-\frac{i}{n\rho}Ke^{\rho u_n}e^{\rho(t_{i-1}-t_{\lfloor \rho\rfloor})}\right]\right)\\
\leq  e&K\|f\|_{\infty} \frac{i e^{\rho    u_n}}{n\rho}\E\left(e^{\rho(t_{i-1}-t_{\lfloor  \rho\rfloor})}\right)\\
=  e&K \|f\|_{\infty} \frac{ie^{\rho    u_n}}{n\rho}
e^{\rho \sum_{k=\lceil \rho \rceil}^{i-1} 1/k}
 e^{\sum_{k=\lceil \rho \rceil}^{i-1}-\log(1-\rho/k)-\rho/k}.
\end{align*}
Thus, there exists some finite  constant $C$ such that, for $i>\rho$,
\[
\E\left(f(i/n^{1/(\rho+2)}, nP_i)\ind{t_{\lfloor \rho\rfloor}\leq u_n}\right)
\leq C \frac{i^{\rho+1} e^{\rho    u_n}}{n}=C \frac{i^{\rho+1} (\log n)^{\rho}}{n},
\]
consequently,
\[
\E\left({\cal Q_{n}(f)}\ind{t_{\lfloor \rho\rfloor}\leq  u_n}\right)
\leq C 
\frac{\beta_n^{\rho+2} (\log n)^{\rho}}{n}
\leq C \frac{1}{(\log n)^{2}}.
\]
This relation and the inequality
\[
\E\left(1-e^{-{\cal Q_{n}(f)}}\right)\leq \P(t_{\lfloor \rho\rfloor}> u_n)+\E\left({\cal Q_{n}(f)}\ind{t_{\lfloor \rho\rfloor}\leq  u_n}\right)
\]
give the desired result. 

\medskip

\medskip

\paragraph{Step 3}
The proof of the theorem can be now completed.  By Equation~\eqref{Pn}, for $i\geq 1$,
$P_i=W_i^\rho Z_i/i^{\rho+1}$,   by using Step~2 and the same  techniques as  in Step~1
together with the fact that, for $\eta>0$, the probability of the event
\[
\left\{\sup\left( \left|{{W}_{i}^\rho}/{{W}_{\beta_n}^\rho}-1\right|:i\geq \beta_n\right)\geq \eta\right\}
\]
converges to $0$ as $n$ gets large, 
it is  not difficult to  show that  the sequences of point  processes 
\[
\left\{\left(\frac{i}{n^{1/(\rho+2)}},\frac{n}{i^{\rho+1}}W_i^\rho Z_i\right), i\geq 1\right\} \text{ and }
\left\{\left(\frac{i}{n^{1/(\rho+2)}},\frac{n}{i^{\rho+1}}W_{\beta_n}^\rho Z_i\right), i\geq \beta_n\right\}
\]
have  the  same  limit  in  distribution.  Because $W_{\beta_n}$ is independent of
$(Z_i,i\geq \beta_n)$,   the last  point  process  has  the  same distribution as ${\cal
  P}_n^2$ (up to the first $\beta_n$ terms which are negligible similarly as in Step~2). By Step~1, the 
convergence in distribution is therefore proved.
\end{proof}
The following proposition strengthens the statement of Proposition~\ref{cvPoispp}, it will
be used to prove the main asymptotic result on the indices of empty bins. 
\begin{prop}\label{diff}
If $f:\R_+^2\to\R_+$ is a continuous function such that
\begin{enumerate}
\item there exists $K$ such that $f(x,y)=0$ for any $x\leq K$ and $y\in\R_+$,
\item for all $x\in\R_+$, the function $y\to f(x,y)$ is differentiable and
\[
y\to y \left\|\frac{\partial f}{\partial y}\right\|_{y}\stackrel{\text{def.}}{=}
y \sup_{x\in\R_+}\left|\frac{\partial f}{\partial y}\right|(x,y)
\]
is integrable on $\R_+$,
\end{enumerate}
then  Convergence~\eqref{LapTrans} also holds for $f$.
\end{prop}
\begin{proof}
For  $M$, $L\geq 0$ and $i$, $n\in\N$, one has 
\begin{multline*}
\E\left(f\left(\frac{i}{n^{\rho+2}},nP_i\right)\ind{nP_i\geq M, t_{\lfloor \rho\rfloor}\leq L}\right)
\\=- \int_0^{+\infty} \frac{\partial f}{\partial y}\left(\frac{i}{n^{\rho+2}},y\right) \P(M\leq nP_i\leq y, t_{\lfloor \rho\rfloor}\leq L)\,dy. 
\end{multline*}
By using similar arguments as in the end of the proof of the above theorem, one gets
\begin{align*}
\E\left(f\left(\frac{i}{n^{\rho+2}},\right.\right.&\left.\left.\rule{-1mm}{5mm}nP_i\right)\ind{nP_i\geq M, t_{\lfloor \rho\rfloor}\leq L}\right)\\
&\leq  e \int_M^{+\infty} \left\|\frac{\partial f}{\partial y}\right\|_y
\E\left(1-\exp\left[-\frac{i}{n\rho}ye^{\rho L}e^{\rho(t_{i-1}-t_{\lfloor  \rho\rfloor})}\right]\right) \,dy \\
&\leq \frac{ie^{\rho L}}{n\rho} e\E\left(e^{\rho(t_{i-1}-t_{\lfloor \rho\rfloor})}\right)
\int_M^{+\infty} y\left\|\frac{\partial f}{\partial y}\right\|_y dy\\
&\leq C\frac{i^{\rho+1}e^{\rho L}}{n} 
\int_M^{+\infty} y\left\|\frac{\partial f}{\partial y}\right\|_y dy,
\end{align*}
for some fixed constant $C$. Define $k_n=\lfloor K n^{1/(\rho+2)}\rfloor$, by summing up
these terms,  this gives the relation
\begin{multline}\label{bb1}
\E\left(\sum_{i\geq 1} f\left(\frac{i}{n^{\rho+2}},nP_i\right)\ind{M\leq nP_i, t_{\lfloor \rho\rfloor}\leq L}\right)\\\leq 
C\frac{k_n^{\rho+2}e^{\rho L}}{n} \int_M^{+\infty} y\left\|\frac{\partial f}{\partial y}\right\|_y dy\leq
CK^{\rho+2}e^{\rho L}\int_M^{+\infty} y\left\|\frac{\partial f}{\partial y}\right\|_y dy.
\end{multline}
Define $f_0(x,y)=f(x,y)\ind{y\leq M}$, by using a convolution kernel on the variable $y$, there exist
sequences $(g_p^+)$ and $(g_p^-)$ in $C_c^+(\R_+)$  converging pointwisely to $f_0$ for
all $y\not=M$ such that $g_p^-\leq f_0\leq  g_p^+$. See Rudin~\cite{Rudin:01} for
example. Proposition~\ref{cvPoispp} gives that  
\begin{multline*}
\E(\exp(-{\cal   P}(g_p^+)))\leq 
\liminf_{n\to+\infty} \E(\exp(-{\cal   P}_n(f_0)))   \\ \leq \limsup_{n\to+\infty} \E(\exp(-{\cal   P_n}(f_0)))
\leq \E(\exp(-{\cal   P}(g_p^-))),
\end{multline*}
and Expression~\eqref{Laplace} shows that, as $p$ goes to
infinity,  the left and right hand side terms of this relation converge to the Laplace
transform of ${\cal P}$  at $f_0$. Therefore, Convergence~\eqref{LapTrans} holds at
$f_0$. Since
\begin{multline*}
0\leq \E\left(e^{-{\cal P}_{n}(f)}\right)-\E\left(e^{-{\cal    P}_{n}(f_0)}\right)
\\\leq P(t_{\lfloor \rho\rfloor}\geq L) +\E\left[\left(1-e^{-({\cal P}_{n}(f)-{\cal P}_{n}(f_0))}\right)\ind{t_{\lfloor \rho\rfloor}\leq L}\right]
\\\leq P(t_{\lfloor \rho\rfloor}\geq L) +\E\left[\left({\cal P}_{n}(f)-{\cal P}_{n}(f_0)\right)\ind{t_{\lfloor \rho\rfloor}\leq L}\right],
\end{multline*}
and the last term being the left hand side of Relation~\eqref{bb1}, one can choose  $L$ and
$M$ sufficiently large so that this difference is arbitrarily small. The proposition is
proved. 
\end{proof}

\section{Asymptotic Behavior of the Indices of the First Empty Bins}\label{seclimit}
It is  assumed that  a large  number $n$ of  balls are  thrown in the  bins according  to the
probability  distribution $(P_i)$  defined  by Equation~\eqref{Pn}.  The  purpose of  this
section is to establish limit theorems to describe the limiting distribution of the set of
indices of bins having a fixed number of balls.
\begin{thm}\label{Res}
The point process of rescaled indices of empty bins associated to the probability vector $(P_i)$
when $n$ balls have been used
\[
{\cal N}_{n}=\left\{\frac{i}{n^{1/(\rho+2)}}: i\geq 1, \text{ the } i\text{th bin  is empty}\right\} 
\]
converges in distribution as $n$ goes to infinity to a point process ${\cal
  N}_{\infty}$ whose distribution is given by 
\begin{equation}
\E\left(e^{-{\cal N}_{\infty}(g)}\right)=
\E\left[\exp\left(-\frac{W_\infty^{-\rho}}{\rho}\int_{\R_+}\left(1-e^{-g(x)} \right)x^{\rho+1}\,dx\right)\right],
\end{equation}
for $g\in C_c^+(\R_+)$. Equivalently $({\cal N}_{n})$ converges in distribution to the point
process
\[
 \left(W_{\infty}^{\rho/(\rho+2)}\,t_i^{1/(\rho+2)}\right),
\]
where $(t_i)$ is a standard Poisson process with parameter $[\rho(\rho+2)]^{-{1}/{(\rho+2)}}$.
\end{thm}
It can also be shown that the same  result holds when the indices of bins containing $k$
balls are considered. If ${\cal N}_{k,n}$ is the corresponding point process, the limiting
point process does not in fact depend on $k$ and, moreover, the sequence $({\cal N}_{k,n},
k\geq 0)$ converges in distribution to $({\cal N}_{k,\infty}, k\geq 0)$ and, conditionally
on $W_\infty$, the random variables  ${\cal N}_{k,\infty}$, $k\geq 0$ are independent with
the same distribution.
\begin{proof}
Poissonization. A closely related model is first analyzed when $U_n$ balls are used, $U_n$
being an  independent Poisson random  variable with mean  $n$, ${\cal N}_n^0$  denotes the
corresponding point  process. For this model,  conditionally on the  sequence $(P_i)$, the
number  of balls  in the  bins  are independent  Poisson random  variables with  respective
parameters $(nP_i)$.  In a  first step,  the convergence in  distribution of  the sequence
$({\cal N}_n^0)$ of point processes is investigated.
Let $g\in C_c^+(\R_+)$,
\[
\E\left(e^{-{\cal N}_n^0(g)}\right)=\E\left(\exp\left[\sum_{i=1}^{+\infty}\log\left[1-e^{-nP_i}\left(1-e^{-g(i/n^{1/(\rho+2)})}\right)\right]\right]\right),
\]
if one defines $f(x,y)=-\log\left[1-e^{-y}\left(1-e^{-g(x)}\right)\right],$
then
\[
\E\left(\exp\left[-{\cal N}_n^0(g)\right]\right)=\E\left(\exp\left[-{\cal P}_{n}(f)\right]\right),
\]
where ${\cal P}_n$ is the point process defined in Theorem~\ref{theoPn}. By using 
Proposition~\ref{diff}, one gets the relation
\begin{align*}
\lim_{n\to+\infty}\E\left(e^{-{\cal N}_n^0(g)}\right)&=
\E\left[\exp\left(-\frac{W_\infty^{-\rho}}{\rho}\int_{\R_+^2}\left(1-e^{-f(x,y)}\right)x^{\rho+1}\,dx\,dy\right)\right]\\
&=
\E\left[\exp\left(-\frac{W_\infty^{-\rho}}{\rho}
\int_{\R_+}\left(1-e^{-g(x)} \right)x^{\rho+1}\,dx\right)\right].
\end{align*}
For $0<\alpha<1$, it is not difficult to check that the same result holds for the case
when $U_{n+n^\alpha}$ balls are used, ${\cal N}_n^1$ denotes the associated point
process. For $x>0$, the monotonicity property  ${\cal N}_a([0,x])\leq {\cal 
  N}_b([0,x])$ for $b\leq a$ gives the relation
\[
\P({\cal N}_n([0,x])\leq k)\leq \P\left({\cal N}_{n}^1([0,x])\leq k\right)+\P(U_{n+n^\alpha}\leq n).
\]
The central limit theorem for Poisson processes shows that for $\alpha\in(1/2,1)$, the
quantity $\P(U_{n+n^\alpha}\leq n)$ converges to $0$ as $n$ gets large, therefore if
$k\geq 0$, 
\[
\limsup_{n\to+\infty}\P({\cal N}_n([0,x])\leq k)\leq \lim_{n\to+\infty}\P\left({\cal N}_{n}^1([0,x])\leq k\right).
\]
By using a similar argument with the $\liminf$, one gets that the sequences $({\cal N}_n)$
and $({\cal N}_n^0)$ converge in distribution and have the same limit. The proposition is
proved. 
\end{proof}

\begin{corollary}
If $\nu_n$ is the index of the first empty bin when $n$ balls are thrown, then
\[
\lim_{n\to+\infty} \P\left(\frac{\nu_n}{n^{1/(\rho+2)}}\geq x\right)=
\E\left(\exp\left(-\frac{x^{\rho+2}W_{\infty}^{-\rho}}{\rho(\rho+2)}\right)\right), \quad
x\geq 0.
\]
\end{corollary}

\medskip

\subsection*{Comparison with Deterministic Power Law Decay}
For  $\delta>1$, one  considers the  bins and  balls problem  with the  probability vector
$Q=(Q_i,  i\geq   1)=\left({\alpha}/{i^{\delta}}\right)$.  Note  that   for  the  problems
analyzed in  this paper, only the  asymptotic behavior of the  sequence $(Q_i)$ matters.
The equivalent  of Theorem~\ref{Res}  can be obtained  directly from Theorem~1  of Simatos
\etal~\cite{Simatos08:0}.
\begin{prop}
As $n$ goes to infinity, the point process
\[
\left\{\frac{i(\log n)^{1/\delta-1}}{(\alpha\delta n)^{1/\delta}}-\frac{1+\delta}{\delta}\log\log n:
\text{ the } i\text{th bin  is empty}\right\}
\]
converges in distribution to a Poisson point process with
the intensity measure $(\alpha\delta)^{1/\delta}e^x\,dx$  on $\R$.
\end{prop}
The probability vector considered in the above theorem has an asymptotic expression of the
form $(P_i)=(W_\infty^\rho E_i/i^{\rho+1})$. In this case, empty bins show up for indices of the
order of  $n^{1/(\rho+2)}$, i.e., much  earlier than for the deterministic case where the exponent of $n$ is $1/\delta=1/(\rho+1)$
(if one ignores the log). This can be explained simply by the fact that some of the
i.i.d.\  exponential random variables $(E_i)$ can be very small thereby creating an additional
possibility of having empty bins. 

In this picture, the variable $W_\infty$ does not seem to have an influence on the
qualitative behavior of these occupancy schemes other than creating some dependency
structure for the vector $(P_i)$. The next section shows that this variable has nevertheless an
important role if one looks at the averages of the number of empty bins. 

\section{Rare Events}\label{secrare}
By Equation~\eqref{LapTrans} of Theorem~\ref{theoPn}, for $x>0$, the limiting number (in
distribution) of empty bins whose index is less than $xn^{1/(\rho+2)}$ has an average
value given by 
\[
\frac{x^{\rho+2}}{\rho(\rho+2)}\E\left(W_\infty^{-\rho}\right)
=\frac{x^{\rho+2}}{\rho(\rho+2)}\int_0^{+\infty} \frac{1}{u^{\rho}}e^{-u}\,du
\]
by Equation~\eqref{Minfty} and since $W_\infty=\exp(-M_\infty)$.This
quantity is infinite when $\rho\geq 1$. The purpose of this section is to investigate this
phenomenon which has a significant impact on the system at the origin of this model. 
It is assumed throughout this section that $\rho\geq 1$. 

\begin{definition}
If $\phi: \N\to\R_+$ is a non-decreasing function, for $n\geq 1$, 
${\cal N}^{\phi}_{n}$ denotes the point process defined by  
\[
{\cal N}^{\phi}_{n}=\left\{\frac{i}{\phi(n)}: i\geq 1, \text{ the } i\text{th bin is empty }\right\}.
\]
\end{definition}
\medskip

\def\rr{{\lfloor \rho\rfloor}}
For $i> \rr$, the quantity $P_i$ can be written as 
$P_i=\exp(-\rho t_\rr)D_iZ_i/i^{\rho+1}$  with
\[
D_i\stackrel{\text{def.}}{=}\exp\left(-\rho\left[M_i-M_\rr -\log \rr \right]\right).
\]
The sequence $(D_i)$ converges almost surely to a finite limit $D_\rho$ given by
\begin{equation}\label{DI}
D_\rho\stackrel{\text{def.}}{=}\exp\left(-\rho\left[M_\infty-M_\rr -\log \rr\right]\right),
\end{equation}
 and, since
$\exp(\rho E_i/i)$ is integrable for $i>\rho$,   a similar result holds for the expected values
\[
\lim_{i\to+\infty}\E(1/D_i)=\E(1/D_\rho)<+\infty. 
\]
With this definition, the asymptotic representation of $(P_i)$ can be given as
$P_i=\exp(-\rho t_\rr)D_\rho E/i^{\rho+1}$ where $E$ is an 
independent exponential random variable with  parameter $1$. In a similar way as before,
this representation can be shown to be valid for the results obtained in this section. 

For $0\leq p\leq 1$ and $n\geq 1$, the elementary inequality
\[
\left|e^{-np}-(1-p)^n\right| \leq \frac{p^2}{2}ne^{-np}\leq\frac{2e^2}{n}
\]
gives directly the following lemma which will be used repeatedly in this section. 
\begin{lemma}
For a non-decreasing function $\phi$, $x\geq 0$, and $n\geq 1$, then
\[
\left|\E\left({\cal N}_n^\phi([0,x])\right)-\sum_{i=1}^{\lfloor x\phi(n)\rfloor}
\E\left(e^{-nP_i}\right)\right|\leq 2e^2\frac{\lfloor  x\phi(n)\rfloor}{n}.
\]
\end{lemma}
When $\phi(n)\ll n$, this lemma implies that to study the asymptotic behavior of
$(\E\left({\cal N}_n^\phi([0,x])\right))$, it is enough to analyze the convergence of the
corresponding sum of the $\E\left(e^{-nP_i}\right)$. 
\noindent
For the moment, $k\in\N$ is fixed, if $n\geq 1$, $i>\rho$, then
\begin{align*}
\E\left(e^{-nP_i}\right)&=\E\left[\exp\left(-nD_\rho e^{-\rho t_{\rr}}E/i^{\rho+1}\right)\right]\\
&=\E\left(\frac{i^{\rho+1}/n}{i^{\rho+1}/n+e^{-t_{\rr}} D_\rho }\right),
\end{align*}
by summing up these terms, if  $\eps_{k,n}\stackrel{\text{def.}}{=}k/n^{1/(\rho+1)}$, one
gets that 
\[
\sum_{i=\rr+1}^k \E\left(
e^{-nP_i}\right)
=n^{1/(\rho+1)} \int_{0}^{\eps_{k,n}}\E\left(\frac{v^{\rho+1}}{v^{\rho+1}+e^{-t_\rr} D_\rho }\right)\,dv+O\left(\eps_{k,n}\right),
\]
which gives the relation
\[
\sum_{i=1}^k \E\left(e^{-nP_i}\right)=n^{1/(\rho+1)}\eps_{k,n}^{\rho+2}
\int_{0}^{1}\E\left(\frac{v^{\rho+1}}{\eps_{k,n}^{\rho+1}v^{\rho+1}+e^{-t_\rr} D_\rho }\right)\,dv +O\left(\eps_{k,n}\right),
\]
with a change of variable. By using Equation~\eqref{eqtn} and again a change of variable,
one obtains the relation
\begin{multline}\label{eg1}
\sum_{i=1}^{k} 
\E\left(e^{-nP_i}\right)
=\frac{\rr}{\rho}n^{1/(\rho+1)}\eps_{k,n}^{{(2\rho+1)}/{\rho}}\\ \times \int_0^{1/\eps^{\rho+1}}\!\!\!\!\!\!
u^{1/\rho-1}(1-\eps_{k,n}^{\frac{\rho+1}{\rho}}u^{1/\rho})^{\rr-1} du\int_{0}^{1}\E\left(\frac{v^{\rho+1}}{v^{\rho+1}+u
    D_\rho }\right)\,dv +O\left(\eps_{k,n}\right).
\end{multline}
This quantity is now analyzed according to the values of $\rho$. 

\bigskip
\noindent
{\bf Case} $\rho>1$.  \\
If $k_n{=}\lfloor xn^{\alpha}\rfloor$ with $1/(2\rho+1)\leq \alpha<1/(\rho+1)$,  then
$\eps_{k_n,n}\sim xn^{(\alpha(\rho+1)-1)/(\rho+1)}$ and, by Relation~\eqref{eg1},
\begin{multline*}
\lim_{n\to+\infty} \frac{1}{n^{((2\rho+1)\alpha-1)/\rho}}\sum_{i=1}^{k_n} \E\left(e^{-nP_i}\right)
\\=x^{(2\rho+1)/\rho} \;\frac{\rr}{\rho} \int_0^{+\infty}
u^{1/\rho-1}du\int_{0}^{1}\E\left(\frac{v^{\rho+1}}{v^{\rho+1}+u
    D_\rho }\right)\,dv.
\end{multline*}

\bigskip
\noindent
{\bf Case }$\mathbf{\rho=1}$.\\
Equation~\eqref{eg1} is for this case 
\[
\sum_{i=1}^{k} 
\E\left(e^{-nP_i}\right)
=\sqrt{n}\eps_{k,n}^{3}\int_0^{1/\eps_{k,n}^2}\,du\int_{0}^{1}\E\left(\frac{v^{2}}{v^{2}+u    D_1 }\right)\,dv +O\left(\eps_{k,n}\right).
\]
If $k_n{=}\lfloor xn^{1/3}/\log^\beta n\rfloor$ with $\beta\in\R$, then
$\eps_{k_n,n}\sim x/(n^{1/6}(\log n)^\beta)$ and for $\beta\leq 1/3$,
\[
\lim_{n\to+\infty} \frac{1}{(\log n)^{(1-3\beta)}}\sum_{i=1}^{k_n} \E\left(e^{-nP_i}\right)
=\frac{1}{9}x^{3}\E\left(\frac{1}{D_1}\right).
\]
The following proposition has therefore been proved. 
\begin{prop}[Average of the Number of Empty Bins]
For  $\alpha$, $\beta>0$, for $n\in\N$, denote by $p_{\alpha,\beta}(n)= n^{\alpha}(\log n)^{-\beta} $,
and by convention $p_\alpha=p_{\alpha,0}$.
\begin{enumerate}
\item If $\rho>1$ and $ 1/(2\rho+1)\leq \alpha<1/(\rho+1)$,
\begin{multline*}
\lim_{n\to+\infty} \frac{1}{n^{{(\alpha(2\rho+1)-1)}/{\rho}}}\E\left( {\cal
  N}^{p_\alpha}_{n}([0,x])\right)
\\= x^{{(2\rho+1)}/{\rho}} \;\frac{\rr}{\rho} \int_0^{+\infty}
u^{1/\rho-1}du\int_{0}^{1}\E\left(\frac{v^{\rho+1}}{v^{\rho+1}+u D_\rho }\right)\,dv.
\end{multline*}
\item If $\rho=1$ and $\beta\leq 1/3$,
\[
\lim_{n\to+\infty} \frac{1}{(\log n)^{(1-3\beta)}}\E\left( {\cal N}^{p_{1/3,\beta}}([0,x])\right)
=\frac{1}{9}x^{3}\E\left(\frac{1}{D_1 }\right)
\]
\end{enumerate}
\end{prop}

\bigskip
\subsection*{A Double Threshold}
For  the  convergence  in  distribution   of  the  sequence  of  point  processes  $({\cal
  N}^\phi_n)$,  Theorem~\ref{Res} has  shown that  the correct  scaling $\phi$  for the
order   of  magnitude   of   the  indices   of  the   first   empty  bins   is  given   by
$\phi(n)=n^{1/(\rho+2)}$, $n\geq 1$.   For the average number of points  in a finite interval,
the   above   proposition states  that,   for  $\rho>1$,   the   correct   scaling   is
in fact $\phi(n)=n^{1/(2\rho+1)}\ll n^{1/(\rho+2)}$. 

For $\alpha>0$, with the notations of the above proposition, one concludes that under the
condition $\rho>1$ and for $1/(2\rho+1) < \alpha< 1/(\rho+2)$, the following limit results hold
\[
{\cal N}_n^{p_\alpha}\stackrel{\text{dist.}}{\to} 0 \text{ and }
\lim_{n\to+\infty} \E\left({\cal N}_n^{p_\alpha}[0,x]\right)=+\infty, \quad \forall x>0.
\]
This suggests that,  in this case, with a  high probability, all the bins  with index less
than $n^{1/(\rho+2)}$ have a  large number of balls. But also that  there exists some rare
event for  which a  very large number of empty bins  with indices of an
order slightly  greater than $n^{1/(2\rho+1)}$ are created.  The following proposition shows  that the
total size  of the first  $\rr$ bins is  the key variable  to explain this  phenomenon. It
should be of  the order of $\log n$ in  order to have sufficiently many  empty bins in the
appropriate region.
\begin{prop}\label{kr}
For $\rho>1$ and if $p_\alpha(n)= n^{\alpha} $, for $\alpha \in [{1}/{(2\rho+1)},{1}/{(\rho+2)})$  and
\[
\delta_0(\alpha)\stackrel{\text{def.}}{=}\frac{1-\alpha(\rho+2)}{\rho-1}
\quad \text{ and }\quad 
\delta_1(\alpha)\stackrel{\text{def.}}{=}\frac{1-\alpha(\rho+1)}{\rho},
\]
then, for  $a\in\R$ and $x>0$, 
\begin{enumerate}
\item If  $\delta<\delta_0(\alpha)$, then
\[
\lim_{n\to+\infty} \E\left({\cal N}^{p_\alpha}_n([0,x])\ind{t_\rr\leq \delta\log n }\right)=0.
\]
\item If  $\delta\in[\delta_0(\alpha),\delta_1(\alpha)[$, then
\[
\lim_{n\to+\infty} \frac{ \E\left({\cal N}^{p_\alpha}_n([0,x])\ind{t_\rr\leq
\delta\log n  +a}\right)}{n^{(\rho+2)\alpha+\delta(\rho-1)-1}}=\frac{x^{\rho+2}}{(\rho+2)}\frac{\rr}{(\rho-1)}\E\left(\frac{1}{D_\rho}\right)e^{(\rho-1)a}.
\]
\item If $\delta \geq \delta_1(\alpha)$,
\begin{multline*}
\lim_{n\to+\infty} \frac{ \E\left({\cal N}^{p_\alpha}_n([0,x])\ind{t_\rr\leq
\delta\log n  +a}\right)}{n^{((2\rho+1)\alpha-1)/\rho}}\\=
x^{(2\rho+1)/\rho} \;\frac{\rr}{\rho} \int_{e^{-\rho a}\ind{\delta=\delta_1(\alpha)}}^{+\infty}
u^{1/\rho-1}du\int_{0}^{1}\E\left(\frac{v^{\rho+1}}{v^{\rho+1}+u
    D_\rho }\right)\,dv,
\end{multline*}
\end{enumerate}
where  $D_\rho$ is the random variable defined by Equation~\eqref{DI}.
\end{prop}
\begin{proof}
To begin with, it is assumed that $\delta\in[\delta_0(\alpha),\delta_1(\alpha))$.
If $k\geq 1$, $b>0$, $\eps_{k,n}=k/n^{1/(\rho+1)}$,  $k=\lfloor x n^{\alpha}\rfloor $  and $b=\delta\log n+a$,
in the same way as for Equation~\eqref{eg1}, one gets 
\begin{align}
\sum_{i=1}^{k} &
\E\left(e^{-nP_i}\ind{t_\rr\leq b}\right)
=\frac{\rr}{\rho}n^{1/(\rho+1)}\eps_{k,n}^{{(2\rho+1)}/{\rho}}\notag \\& \times
\int_{e^{-\rho b}/\eps_{k,n}^{\rho+1}}^{1/\eps_{k,n}^{\rho+1}}
u^{1/\rho-1}(1-\eps_{k,n}^{\frac{\rho+1}{\rho}}u^{1/\rho})^{\rr-1} du\int_{0}^{1}\E\left(\frac{v^{\rho+1}}{v^{\rho+1}+u
    D_\rho }\right)\,dv+O(\eps_{k,n})\notag \\
&=\frac{\rr}{\rho}n^{\frac{1}{\rho+1}}\eps_{k,n}^{\frac{2\rho+1}{\rho}}
\int_{e^{-\rho b}/\eps_{k,n}^{\rho+1}}^{1/\eps_{k,n}^{\rho+1}}\!\!\!\!\!\!
u^{1/\rho-2}du\int_{0}^{1}\E\left(\frac{uv^{\rho+1}}{v^{\rho+1}+u
    D_\rho }\right)\,dv{+}O(\eps_{k,n}).\label{eg2}
\end{align}
Note that 
\[
{e^{-\rho b}}/{\eps_{k,n}^{\rho+1}}
\sim n^{1-\rho\delta-\alpha(\rho+1)}e^{-\rho a}\nearrow +\infty,
\]
hence the range of the first integral goes to infinity as $n$ gets large. 
Since
\[
\int_{0}^{1}\E\left(\frac{uv^{\rho+1}}{v^{\rho+1}+u D_\rho }-\frac{v^{\rho+1}}{ D_\rho}\right)\,dv
=\int_{0}^{1}\E\left(\frac{v^{2(\rho+1)}}{(v^{\rho+1}+u D_\rho)D_\rho }\right)\,dv,
\]
by Lebesgue's Theorem, this integral is arbitrarily small as $u$ gets large, this implies
the equivalence
\[
\sum_{i=1}^{k} \E\left(e^{-nP_i}\ind{t_\rr\leq b}\right)
\sim\frac{\rr}{\rho(\rho+2)}\E\left(\frac{1}{D_\rho }\right)n^{1/(\rho+1)}\eps_{k,n}^{{(2\rho+1)}/{\rho}}
\int_{e^{-\rho b}/\eps_{k,n}^{\rho+1}}^{1/\eps_{k,n}^{\rho+1}}
u^{1/\rho-2}du.
\]
If $C$ is the multiplicative constant of the right hand side of the above relation, then
\[
\sum_{i=1}^{k} \E\left(e^{-nP_i}\ind{t_\rr\leq b}\right)\sim \frac{C\rho}{\rho-1}
\frac{k^{\rho+2}}{n}\left(e^{b(\rho-1)} -1\right),
\]
this gives the equivalence
\[
\sum_{i=1}^{k} \E\left(e^{-nP_i}\ind{t_\rr\leq b}\right)\sim x^{\rho+2}
\frac{C\rho}{\rho-1} e^{a(\rho-1)}n^{(\rho+2)\alpha+\delta(\rho-1)-1}.
\]
The proof of this case is completed. 

The case $\delta\geq \delta_1(\alpha)$ uses Equation~\eqref{eg2}. The term ${e^{-\rho
 b}}/{\eps_{k,n}^{\rho+1}}$ converges to $e^{-\rho a}$  if $\delta=\delta_1(\alpha)$ and
  $0$ otherwise. This gives directly the desired convergence. 

Finally, if $\delta<\delta_0(\alpha)$, for any $a\in\R$, there exists $n_0$ so that if
$n\geq n_0$, then $\delta\log n\leq \delta_0(\alpha)\log n  +a$, in particular
\[
 \E\left({\cal N}^{p_\alpha}_n([0,x])\ind{t_\rr\leq \delta\log n }\right)
\leq 
\E\left({\cal N}^{p_\alpha}_n([0,x])\ind{t_\rr\leq \delta_0(\alpha)\log n  +a}\right)
\]
hence
\[
\limsup_{n\to+\infty}  \E\left({\cal N}^{p_\alpha}_n([0,x])\ind{t_\rr\leq \delta\log n}\right)
\leq 
\frac{x^{\rho+2}}{(\rho+2)}\frac{\rr}{(\rho-1)}\E\left(\frac{1}{D_\rho}\right)e^{(\rho-1)a}.
\]
One concludes by letting $a$ go to $-\infty$.
\end{proof}
As     a    consequence     of    the     above    proposition,     for     $\alpha    \in
[{1}/{(2\rho+1)},{1}/{(\rho+2)})$,     the    average     of    the     variable    ${\cal
      N}^{p_\alpha}_n([0,x])$ converges  to infinity only  when the total size  $t_\rr$ of
    the  first  $\rr$ bins  is  of  the order  $\delta\log  n$  for  a sufficiently  large
    $\delta$. The following corollary gives a more precise formulation.
\begin{corollary}\label{corolkr}
For $\rho>1$ and if $p_\alpha(n)= n^{\alpha} $, for $\alpha \in
[{1}/{(2\rho+1)},{1}/{(\rho+2)})$ 
\[
\delta_1(\alpha){=}{(1-\alpha(\rho+1))}/{\rho},
\]
then, for $a$, $b>0$,
\[
\lim_{n\to+\infty} \frac{\E\left({\cal N}^{p_\alpha}_n([0,x])\ind{\delta_1(\alpha)\log n
    -a \leq t_\rr\leq \delta_1(\alpha)\log n  +b}\right)}{\E\left({\cal
    N}^{p_\alpha}_n([0,x])\right)}=\psi(-a,b)
\]
where, for $y$, $z\in\R$, $\psi(y,z)=\phi(y,z)/\phi(-\infty,+\infty)$ and
\[
\phi(y,z)=x^{(2\rho+1)/\rho} \;\frac{\rr}{\rho} \int_{[e^{-\rho z},e^{-\rho y}]}
u^{1/\rho-1}du\int_{0}^{1}\E\left(\frac{v^{\rho+1}}{v^{\rho+1}+u
    D_\rho }\right)\,dv.
\]
\end{corollary}
A rough (non-rigorous) interpretation of this result could be as follows: on the event where ``most''
(i.e., for the averages)  of empty bins are created in the interval $[0,xn^\alpha]$, the random variable $t_\rr -
\delta_1(\alpha)\log n$ converges in distribution to some random variable $X$ on $\R$,
such that $\P(X\leq a)=\psi(-\infty,a)$.  

The following analogous result is proved in a similar way for the critical case $\rho=1$. 
\begin{prop}
For $\rho=1$ and with the notations of the above proposition
then, for   $0<\beta<1/3$,  $x>0$, and for $0\leq a\leq 1/3$,
\[
\lim_{n\to+\infty} \frac{1}{(\log n)^{1-3\beta}}\E\left({\cal N}^{p_{1/3,\beta}}_n([0,x])\ind{t_\rr\leq a\log n}\right)=\frac{a}{3}x^3\E\left(\frac{1}{D_1}\right),
\]
and for $a>1/3$,
\[
\lim_{n\to+\infty} \frac{1}{(\log n)^{1-3\beta}}\E\left({\cal N}^{p_{1/3,\beta}}_n([0,x])\ind{t_\rr\leq a\log n}\right)=\frac{1}{9}x^3\E\left(\frac{1}{D_1}\right),
\]
where  $D_1$ is the random variable defined by Equation~\eqref{DI}.
\end{prop}

\section{Generalizations}\label{secgen}
The problem analyzed in the present paper can be generalized towards two directions. On
one hand, the sequence $(t_n)$ can stem from a general branching process instead of the
particular Yule one; on the other hand, the locations of balls can have a general
distribution. This section discusses these possible extensions. 

\subsection*{Exponential Balls and General Branching Process}
Let $(t_n)$ be  the birth instants of a general  supercritical branching process $(Z(t))$.
See Kingman~\cite{Kingman:02}  and Nerman~\cite{Nerman:01} for example.  Let $\alpha$ be
the Malthusian parameter, and $W$ the  almost sure limit of $(e^{-\alpha t} Z(t))$.  Under
reasonable  technical  assumptions,  H\"arnqvist~\cite{Harnqvist81:0}  has shown  the  following
result:
\begin{thm}
Define the point process $\Psi_t^*$ by
\[
\Psi_t = \sum_{k \geq 1} \ind{t \leq t_k}\delta_{t_ke^{\alpha t}},
\]
as $t$ gets large, $\Psi_t$ converges in distribution to a mixed Poisson process
whose parameter is distributed as $\gamma W$ for some constant $\gamma > 0$. 
\end{thm}
From this result, it is possible to prove that the process $(n(t_{n+k} - t_n), k \geq 1)$
converges in distribution, as $n$ goes to infinity, to a Poisson process: clearly
\[
\Psi_{t_n} = \sum_{k\geq 1} \delta_{(t_{n+k} - t_n) e^{\alpha t_n}},
\]
and provided that, up to a multiplicative constant, $e^{\alpha t_k} / k$ converges to $W$, the point process
$\sum_{k \geq 1} \delta_{n(t_{n+k} - t_n) }$
should converge to a Poisson random variable with a deterministic parameter. In this case
the probability that a ball falls into the $n$th bin which is given by $$P_n = e^{-\rho t_{n-1}} (
1- e^{-\rho (t_n-t_{n-1})}),$$ has therefore the following asymptotic behavior 
\[
P_n \sim n^{-\rho / \alpha} W^{\rho / \alpha} E_i,
\]
where $(E_i)$ are i.i.d.\ exponential random variables. In the Bellman-Harris case,
following Athreya and Kaplan~\cite{Athreya76:0}, it is possible to show that  $W$ and $(E_i)$ are
independent, so that in this case, the asymptotic behavior of $(P_n)$ is exactly the same
as in the case of a Yule process. One can conjecture that this independence property still
holds in the general case. 

The main  obstacle to  generalize the results  of this  paper, even in  the Bellman-Harris
case, is that although $W$ and $(E_i)$  are independent, $t_{n-1}$ and $t_n - t_{n-1}$ are
not independent. In the proof of Proposition~1, this independence plays a crucial role, it
has   therefore  to   be  generalized   to   variables  which   are  only   asymptotically
independent.  Additionally,  since  the  heavy  tail property  of  the  limiting  variable
$W_\infty^{-\rho}$ is also true in the  general case, see e.g., Liu~\cite{Liu01:0}, a
similar rare events phenomenon to the one described in Section~\ref{secrare} is plausible
in this case. 


\subsection*{General Balls and Yule Process}
When the underlying branching  process is changed, the above discussion suggests that the
asymptotic  behavior of the sequence $(P_n)$ remains essentially the same as for a Yule
process. The situation changes significantly when the
law of the location $X$ of a ball is changed, in this case with the same notations as
before for the Yule process,
\[
	P_n = \P(t_{n-1} < X \leq t_{n-1} + E_n / n).
\]
The tail distribution of $X$ then plays a key role. Consider for instance a
power law, i.e., $\P(X\geq x)$ behaves as $\delta x^{-\beta}$ for some $\beta$ and $\delta > 0$: then 
\[
P_{n+1} \sim t_{n}^{-\beta} - (t_{n} + E_{n+1} /{(n+1)})^{-\beta} \sim \frac{\beta\delta
  E_{n+1}}{n t_n^{\beta+1}} \sim \frac{\beta\delta E_n}{n (\log n)^{\beta+1}},
\]
and it can be seen that the random variable $W_\infty$ may not play a role anymore in the asymptotic
behavior of the system.

\end{document}